\documentclass[11pt,
twoside,%
onecolumn,%
nofloats,%
nobibnotes,%
nofootinbib,%
superscriptaddress,%
noshowpacs,%
centertags]%
{revtex4}
\usepackage{ljm1}
\usepackage{enumerate}
\usepackage{upgreek}

\theoremstyle{definition}
\newtheorem{remark}{Remark} 

\begin{document}
\titlerunning{Integrals of the difference of subharmonic functions\dots}
\authorrunning{B. N. Khabibullin}

\title{Integrals of the difference of subharmonic functions\\ over discs and planar small sets}

\author{\firstname{B. N.}~\surname{Khabibullin}}
\email[E-mail: ]{khabib-bulat@mail.ru } 
\affiliation{Bashkir State University, Zaki Validi Str. 32, Bashkortostan, 420076 Russian Federation}



\received{November 25, 2020}

\begin{abstract}
The maximum of the modulus of a meromorphic function cannot be restricted from above by the Nevanlinna characteristic of this meromorphic function. But integrals from the logarithm of the module of a meromorphic function allow similar restrictions from above. This is illustrated by one of the important  theorems of Rolf Nevanlinna in the classical monograph by A. A. Goldberg and I. V. Ostrovskii on meromorphic functions, as well as by the Edrei\,--\,Fuchs Lemma on small arcs and its versions for small intervals in articles by  A.\,F.~Grishin, M.\,L.~Sodin, T.\,I.~Malyutina.  Similar results for integrals of differences of subharmonic functions even with weights were recently obtained by B.N. Khabiblullin, L.A. Gabdrakhmanova. All these results are on  integrals over subsets on a ray. In this article, we establish such results for integrals of the logarithm of the modulus of a meromorphic function and the difference of subharmonic functions over discs and planar small sets. Our estimates are uniform in the sense that the constants in these estimates are explicitly written out and  do not depend on meromorphic functions and  the difference of subharmonic functions  provided that these functions has an integral normalization near zero.
\end{abstract}



\maketitle

\section{Introduction. Definitions and notation}
Upper estimates of integrals for meromorphic functions in the {\it complex plane\/} ${\mathbb{C}}$ over segments or small sets on a ray via the Nevanlinna characteristic are considered in works \cite[pp. 24--27]{RNevanlinna}, \cite[Notes, Ch. 1]{GOe}, 
\cite[Ch.~1, Theorem~7.2]{GOe}, \cite[Lemma~3.1]{GrS}, \cite[Theorem~8]{GrM}, 
\cite[Theorem~1 (on small intervals), Remark~1.1, Conclusion of the Grishin\,--\,Malyutina  Theorem]{GabKha20}, 
\cite[Main Theorem]{KhaMst}. Our article \cite{KhaMst} contains all these listed results as special cases (see  Theorem \ref{th1l} below). We give in \cite[Inrtoduction]{GabKha20} and \cite[1.1--1.2]{KhaMst} a detailed history of the issue with full formulations of all previous results. In this paper, we obtain similar upper estimates  for integrals already over discs or small planar sets. As in [1], we establish the main theorem of the paper in a more general subharmonic version. Let's move on to precise definitions.

As usual, 
$\mathbb R$  is the \textit{real line,\/} or the {\it real axis\/} of the complex plane $\mathbb C$, and  ${\mathbb{R}}^+:=\{r\in {\mathbb{R}}\colon 0\leq r\}$ is the {\it positive closed semiaxis.\/}  We denote singleton sets by a symbol without curly brackets. So,  ${\mathbb{R}}^+\!\setminus\!0$ is the {\it positive open semiaxis,\/}
$\overline {\mathbb{R}}$   is the {\it extended real axis.\/}
Besides, $ D(z,r):=\{z' \in {\mathbb{C}} \colon |z'-z|<r\}$ is an {\it open disc,\/} $\overline  D(z,r):=\{z' \in {\mathbb{C}} \colon |z'-z|\leq r\}$ 
is a {\it closed disc,} 
$\partial \overline D(z,r):=\{z' \in {\mathbb{C}} \colon |z'-z|=r\}$ is a  
{\it circle with center\/ $z\in {\mathbb{C}}$ of radius\/ $r\in \overline {\mathbb{R}}^+$;}  $D(z,0)=\varnothing$, $\overline  D(z,0)=\partial \overline D(z,0)=z$,  $D(z,+\infty)={\mathbb{C}}$, $D(r):=D(0,r), \overline D(r):= \overline D(0,r)$, 
$\partial \overline D(r):=\partial \overline D(0,r)$. 

Given a function  $f\colon X\to {\overline {\mathbb{R}}}$,    $f^+:=\sup\{0,f\}$ and  $f^-:=(-f)^+$ are  \textit{positive\/} and \textit{negative  parts\/} of function $f$, respectively; $|f|:=f^++f^-$.  

Given  $S\subset {\mathbb{C}}$, ${\sf sbh}(S)$ is the class of  all {\it subharmonic\/}  on an open  neighbourhood of $S$.  We set   ${\sf sbh}_*(S):=\bigl\{v\in {\sf sbh}(S)\colon v\not\equiv -\infty\bigr\}$.

By $\uplambda$ we denote the {\it planar Lebesgue measure\/} on ${\mathbb{C}}$. We also use the notation ${\rm mes} $ 
for the {\it linear Lebesgue measure\/} on ${\mathbb{R}}$. 

For $ r\in {\mathbb{R}}^+ $ and a function $ v\colon \partial \overline D(r) \to {\overline {\mathbb{R}}}$, we define
\begin{subequations}\label{Cu}
\begin{align}
{\sf M}_v(r)&:=\sup_{|z|=r}v(z),\quad r\in {\mathbb{R}}^+,  
\tag{\ref{Cu}M}\label{{Cu}M}\\
{\sf C}_v(r)&:=\frac{1}{2\pi}\int_0^{2\pi} v(re^{i\varphi}){\,{\rm d}} \varphi ,  \quad r\in {\mathbb{R}}^+\!\setminus\!0,
\tag{\ref{Cu}C}\label{{Cu}C}
\end{align}
\end{subequations}
where ${\sf C}_v(r)$ is the {\it average over\/} the {\it circle\/} $\partial\overline D(0, r)$, if this integral exists.
For  a $\uplambda$-measurable  function $v\colon \overline D(r)\to \overline {\mathbb{R}}$, we use also the {\it average over\/} the {\it disc\/} $\overline D(r)$ 
 defined as 
\begin{equation}\label{Bu}
{\sf B}_v(r)
:=\frac{1}{\pi r^2}\int_{\overline D(r)}v {\,{\rm d}} \uplambda ,  
\quad {\sf B}_v(0):={\sf M}_v(0)\overset{\eqref{{Cu}M}}{=}v(0)=:{\sf C}_v(0).
\end{equation}
See \cite[2.6]{Rans}, \cite[2.7]{HK}, \cite[3]{BaiKhaKha17}  on properties of  ${\sf M}_v$, ${\sf C}_v $, ${\sf B}_v $  in the case of a subharmonic function $v$.

For a Borel subset $S\subset {\mathbb{C}}$, the set  of all Borel, or Radon,  positive measures $\mu\geq 0$  on $S$
is denoted by ${\sf Meas\,}^+(S)$, 
and ${\sf Meas\,}(S):={\sf Meas\,}^+(S)-{\sf Meas\,}^+(S)$ is the set  of all {\it charges,\/} or signed measures, on $S$.
For a charge $\nu\in {\sf Meas\,}(S)$, we denote by 
\begin{equation*}
\nu^+\in {\sf Meas\,}^+(S), \quad \nu^-:=(-\nu)^+\in {\sf Meas\,}^+(S), \quad |\nu|:=\nu^++\nu^-\in {\sf Meas\,}^+(S)
\end{equation*} 
the {\it upper, lower, total variations\/} of this charge $\nu$, respectively.      
We set
\begin{subequations}\label{murad}
\begin{flalign}
\nu(z, r)&:=\nu\bigl(\overline D(z,r)\bigr)\in \overline {\mathbb{R}} \quad\text{if } \overline D(z,r)\subset S,
\quad 0\leq r\leq R,
\tag{\ref{murad}z}\label{{murad}z}
\\
\nu^{{\text{\tiny\rm rad}}}(r)&:=\nu(0,r)=\nu\bigl(\overline D(r)\bigr)\in \overline{\mathbb{R}}  \quad\text{if }\overline D(r)\subset S,
\quad 0\leq r\leq R,
\tag{\ref{murad}r}\label{{murad}m}
\\
{\sf N}_{\nu}(r,R)&:=\int_{r}^{R}\frac{\nu^{{\text{\tiny\rm rad}}}(t)}{t}{\,{\rm d}} t\in \overline {\mathbb{R}}^+ 
 \quad\text{if } \overline D(R)\subset S,
\quad 0\leq r\leq R,
\tag{\ref{murad}N}\label{{murad}N}
\end{flalign}
\end{subequations} 
provided that the last integral is well defined. For $ 0\leq r\leq R\in {\mathbb{R}}^+ $ and  functions 
$$
v\colon \partial\overline D(r)\cup \partial\overline D(R)\to {\overline {\mathbb{R}}},
$$ we define
\begin{equation}\label{MCr}
{\sf C}_v(r,R)\overset{\eqref{{Cu}C}}{:=}{\sf C}_v(R)-{\sf C}_v(r)=\frac{1}{2\pi}\int_0^{2\pi} \bigl(v(Re^{i\varphi})-v(re^{i\varphi})\bigr){\,{\rm d}} \varphi 
\end{equation}
provided that ${\sf C}_v(R)$ and ${\sf C}_v(r)$ are well defined.

If $D\subset {\mathbb{C}}$ is a domain and  $u\in {\sf sbh}_*(D)$, then there is  its {\it Riesz measure\/} 
\begin{equation}\label{df:cm}
\varDelta_u:= \frac{1}{2\pi} {\bigtriangleup}  u\in {\sf Meas\,}^+( D), 
\end{equation}
where ${\bigtriangleup}$  is  the {\it Laplace operator\/}  acting in the sense of the  theory of distribution or generalized functions. This definition of the Riesz measures carries over naturally to $u\in {\sf sbh}_*(S)$ for connected subsets $S\subset {\mathbb{C}}$.
By the Poisson\,--\,Jensen\,--\,Privalov formula \cite{Rans}, \cite{HK},  we have
\begin{equation}\label{CN}
{\sf C}_v(r,R)={\sf N}_{\varDelta_v}(r,R)
\quad\text{for  all $0<r<R<+\infty$ if $v\in {\sf sbh}_*\bigl(\overline D(R)\bigr)$}.
\end{equation}

Let $U=u-v$ be a difference of subharmonic functions $u,v \in {\sf sbh}_*\bigl(\overline D(0, R)\bigr)$, i.\,e., a  {\it $\delta$-subharmonic non-trivial ($\not\equiv\pm\infty$) function\/} \cite{Arsove53}, \cite{Arsove53p}, \cite{Gr}, \cite[3.1]{KhaRoz18} on $\overline D(R)$ with the {\it Riesz charge\/}
\begin{equation*}
\varDelta_{U}\overset{\eqref{df:cm}}{:=}\varDelta_u-\varDelta_v\overset{\eqref{df:cm}}{:=}
\frac{1}{2\pi}{\bigtriangleup} u-\frac{1}{2\pi}{\bigtriangleup} v\in {\sf Meas\,} \bigl(\overline D(0, R)\bigr), \text{ and }
\varDelta^-_U:=(\varDelta_U)^-
\end{equation*} 
is the lower 
variation of the Riesz charge  $\varDelta_U$ of $U$. 
Now we can determine the {\it difference  Nevanlinna characteristic\/} ${\sf T}$  of 
$\delta$-subharmonic non-trivial ($\not\equiv\pm\infty$) function $U$ as a function of two variables 
\begin{equation}\label{T0}
{\sf T}_U(r,R)
:={\sf C}_{U^+}(r,R)+{\sf N}_
{\varDelta_U^-}(r,R), \quad 0<r\leq R\in {\mathbb{R}}^+. 
\end{equation}
A representation $U=u_U-v_U$ with $u_U,v_U\in {\sf sbh}_*\bigl(\overline D(0, R)\bigr)$ is {\it canonical\/} if the  Riesz measure $\varDelta_{u_U}$ of $u_U$ is  the {\it upper variation\/}  $\varDelta_U^+$ of $\varDelta_U$  and  the Riesz measure  $\varDelta_{v_U}$ of $v_U$ is  the {\it lower  variation\/}  $\varDelta_U^-:=(\varDelta_U)^-$ of $\varDelta_U$. 
The canonical representation for $U$ is defined up to the harmonic function added simultaneously to each of the representing subharmonic functions $u_U$ and $v_U$, and 
\begin{equation}\label{T}
{\sf T}_U(r,R)
\overset{\eqref{CN},\eqref{T0}}{=}{\sf C}_{U^+}(r,R)+{\sf C}_{v_U}(r,R)={\sf C}_{\sup\{u_U,v_U\}}(r,R), 
\quad 0<r\leq R\in {\mathbb{R}}^+. 
\end{equation}
By \eqref{T}, the difference Nevanlinna characteristic $T_U$ is already uniquely defined for all values $0<r\leq R<+\infty$  by positive values in ${\mathbb{R}}^+$, and is also  increasing and convex with respect to the logarithmic function $\ln$ in the second variable $R$, but is decreasing  in the first  variable $r\leq R$. 
Recall that the following notation is used for the meromorphic function
$F\not\equiv 0$  on ${\mathbb{C}}$ in the classic monograph by 
A. A. Goldberg and I. V. Ostrovskii \cite{GOe} for the maximum of module  
\begin{subequations}\label{TN}
	\begin{align}
	 M(r,F)&\underset{\text{\tiny$r\in{\mathbb{R}}^+$}}{:=}\sup\bigl\{\bigl|F(z)\bigr| \colon |z|=r\bigr\}, 
	\tag{\ref{TN}M}\label{{TN}M}\\
	\intertext{and for  the {\it Nevanlinna characteristic}}
	T(r, F)&\underset{\text{\tiny $r\in{\mathbb{R}}^+$}}{:=}m(r,F)+N(r,F),  
	\tag{\ref{TN}T}\label{{TN}T}
	\\
	m(r,F)&\underset{\text{\tiny$r\in{\mathbb{R}}^+$}}{:=}\frac{1}{2\pi}\int_0^{2\pi} \ln^+\bigl|F(re^{i\varphi})\bigr| {\,{\rm d}} \varphi, 
	 \tag{\ref{TN}m}\label{{TN}m}\\
	N(r,F)&\underset{\text{\tiny $r\in{\mathbb{R}}^+$}}{:=}\int_{0}^{r}\frac{n(t,F)-n(0,F)}{t}{\,{\rm d}} t+n(0,F)\ln r, 
	\tag{\ref{TN}N}\label{{TN}N}
	\end{align}
\end{subequations} 
where   $n(r,F)$ is the number of poles of $F$ in the closed disc $\overline D(r):=\{z\in {\mathbb{C}}\colon |z|\leq r\}$, taking into account the multiplicity. The function $\ln |F|$ is non-trivial $\delta$-subharmonic on ${\mathbb{C}}$, and 
\begin{subequations}\label{cs}
\begin{align}
\ln M(r, F)&\overset{\eqref{{TN}M},\eqref{{Cu}M}}{=}{\sf M}_{\ln|F|}(r), \quad r\in {\mathbb{R}}^+,
\tag{\ref{cs}M}\label{{cs}M}\\  
m(r, f)&\overset{\eqref{{TN}m},\eqref{{Cu}C}}{=}{\sf C}_{\ln^+|F|}(r),\quad r\in {\mathbb{R}}^+,
\tag{\ref{cs}m}\label{{cs}m}\\
N(R, F)-N(r, F)&\overset{\eqref{{TN}N},\eqref{{murad}N}}{=}{\sf N}_
{\varDelta_{\ln|F|}^-}(r,R),
\quad 0<r<R\in {\mathbb{R}}^+,
\tag{\ref{cs}N}\label{{cs}N}\\
T(R, F)-T(r, F)&\overset{\eqref{{TN}T},\eqref{T0}}{=}{\sf T}_{\ln|F|}(r,R).
\quad 0<r<R\in {\mathbb{R}}^+,
\tag{\ref{cs}T}\label{{cs}T}
\end{align}
\end{subequations}

\section{Recent and new results}

Let us formulate the main result from \cite{KhaMst} for the ``one-dimensional'' subset $E \subset [0,r]\subset {\mathbb{R}}^+$ in the special case without a weight function-multiplier $g\in  L^p(E)$, $1<p\leq+\infty$, in the integrand.

\begin{theorem}[{\cite[Main Theorem]{KhaMst}}]\label{th1l} 
Let\/  $0< r_0< r<+\infty$, $1<k\in {\mathbb{R}}^+$,  $E\subset [0,r]$ be ${\rm mes}$-me\-a\-s\-u\-r\-a\-b\-le, 
$g\in L^p(E)$, where $1<p\leq \infty$ and $q\in [1,+\infty)$ is defined by $\dfrac{1}{p}+\dfrac{1}{q}=1$.
If $U\not\equiv \pm\infty$  is a non-trivial $\delta$-subharmonic   functions on ${\mathbb{C}}$, 
and $u\not\equiv -\infty$ is a subharmonic function on ${\mathbb{C}}$,  then 
\begin{subequations}\label{1l}
\begin{flalign}
\int_{E} {\sf M}_{U^+}(t)g(t){\,{\rm d}} t&\leq
\frac{4qk}{k-1} \bigl({\sf T}_{U}(r_0,kr)+{\sf C}_{U^+}(r_0)\bigr)\|g\|_{L^p(E)} \sqrt[q]{{\rm mes} \,E}
\ln\frac{4kr}{{\rm mes} \,E},
\tag{\ref{1l}T}\label{inDl+l}
\\ 
\int_{E} {\sf M}_{|u|}(t)g(t){\,{\rm d}} t& \leq
\frac{5qk}{k-1} \bigl({\sf M}_{u^+}(kr)+{\sf C}_{u^-}(r_0)\bigr) \sqrt[q]{{\rm mes} \,E}
\ln\frac{4kr}{{\rm mes} \,E}.
\tag{\ref{1l}M}\label{uMl}
\end{flalign}
\end{subequations}
\end{theorem}
In particular,  by \eqref{{cs}M}--\eqref{{cs}T}   we have for meromorphic functions the following
\begin{corollary} Let, under the conditions of Theorem\/ {\rm \ref{th1l}}, $F\not\equiv 0$ be a meromorphic function, and $f\not\equiv 0$ be an entire function.
Then, in the traditional notation  \eqref{{TN}M}--\eqref{{TN}N}, we have
\begin{subequations}\label{1lf}
\begin{flalign}
\int_{E} \bigl(\ln^+ M(t,F)\bigr)g(t){\,{\rm d}} t&\leq
\frac{4qk}{k-1} \bigl( T(kr,F)-N(r_0,F)\bigr)\|g\|_{L^p(E)} \sqrt[q]{{\rm mes}\,E}
\ln\frac{4kr}{{\rm mes} \,E},
\tag{\ref{1l}T}\label{inDl+lf}
\\ 
\int_{E} \bigl|\ln M(t,f)\bigr|  
g(t){\,{\rm d}} t& \leq
\frac{5qk}{k-1} \bigl(\ln^+ M(kr,f)+m(r_0,1/f )
\bigr) \sqrt[q]{{\rm mes} \,E}
\ln\frac{4kr}{{\rm mes} \,E}.
\tag{\ref{1l}M}\label{uMlf}
\end{flalign}
\end{subequations}
\end{corollary}
\begin{proof} By \eqref{{cs}M}--\eqref{{cs}T} we have for $U:=\ln |F|$ 
\begin{equation}\label{TT}
{\sf T}_{U}(r_0,kr)+{\sf C}_{U^+}(r_0)\overset{\eqref{{cs}T},\eqref{{cs}m}}{=}T(kr,F)-T(r_0,F)+
m(r_0,f)\overset{\eqref{{TN}T}}{=}T(kr,F)-N(r_0,F).
\end{equation}
We can replace the bracket on the right-hand side of \eqref{inDl+l} with the right side of these equalities. 
Then we obtain  \eqref{inDl+lf} by \eqref{{cs}M} for the integrand in the left-hand side of \eqref{inDl+l}. 

For $u:=\ln |f|$, we have 
\begin{equation}\label{MM}
{\sf M}_{u^+}(kr)+{\sf C}_{u^-}(r_0)\overset{\eqref{{cs}M}}{=}\ln^+ M(kr,f)+
{\sf C}_{\ln^-|f|}(r_0)\overset{\eqref{{cs}m}}{=}\ln^+ M(kr,f)+m(r_0,1/f ). 
\end{equation}
We can replace the bracket on the right-hand side of \eqref{uMl} with the right side of these equalities. 
Then we obtain  \eqref{uMlf} by \eqref{{cs}M} for the integrand in the left-hand side of \eqref{uMl}. 
\end{proof}
\begin{remark}
Our elementary example \cite[1.1, (3)]{KhaMst} shows that it is impossible to discard the terms $-N(r_0,F)$ and 
$m(r_0,1/f )$ in parentheses on the right-hand sides of inequality \eqref{inDl+lf} and \eqref{uMlf}, respectively.
In particular, the classical Rolf Nevanlinna Theorem \cite[Ch.~1, Theorem 7.2]{GOe}  can be formulated in the following correct form: {\it for each meromorphic function   $F\not\equiv 0$ and   $1<k\in {\mathbb{R}}^+$}
\begin{equation}
\frac{1}{r}\int_0^r \ln^+ M(t,F){\,{\rm d}} t\leq
\frac{4k\ln 4k}{k-1} \bigl( T(kr,F)-N(r_0,F)\bigr)\quad \text{\it for all\/ $0<r_0\leq r\in {\mathbb{R}}^+$.}
 \end{equation}
Indeed, it is sufficient to choose $E:=[0,r]$, $g\equiv 1$ with ${\rm mes} \,E=r$, and $p:=\infty$ with $q=1$ in \eqref{inDl+lf}.
\end{remark}
Our main result is established for the case of a planar ``two-dimensional'' subset 
$E\subset \overline D(r)\subset {\mathbb{C}}$:
\begin{theorem}\label{th1} 
Let  $0< r_0< r<+\infty$, $1<k\in {\mathbb{R}}^+$,  $E\subset \overline D(r)$ be a $\uplambda$-measurable subset. If  
$U\not\equiv \pm\infty$  is a $\delta$-subharmonic  functions on ${\mathbb{C}}$, 
and $u\not\equiv -\infty$ is a subharmonic function on ${\mathbb{C}}$,  then 
\begin{subequations}\label{1}
\begin{flalign}
\int_{E} U^+{\,{\rm d}} \uplambda \leq
\frac{2k}{k-1} \bigl({\sf T}_{U}(r_0,kr)+{\sf C}_{U^+}(r_0)\bigr)\uplambda(E)
\ln\frac{100kr^2}{\uplambda (E)},
\tag{\ref{1}T}\label{inDl+}
\\ 
\int_{E} |u|{\,{\rm d}} \uplambda \leq
\frac{3k}{k-1} \bigl({\sf M}_{u^+}(kr)+{\sf C}_{u^-}(r_0)\bigr) \uplambda(E) \ln\frac{100kr^2}{\uplambda (E)}.
\tag{\ref{1}M}\label{uM}
\end{flalign}
\end{subequations}
\end{theorem}
Theorem  \ref{th1}  is proved at the end of Sec.~\ref{Sec3} after some preparation. 

We have not seen before these estimates \eqref{inDl+} and \eqref{uM}  even for the case $E = \overline D(r)$:
\begin{corollary}  If $U\not\equiv \pm\infty$  be a $\delta$-subharmonic  functions on ${\mathbb{C}}$, and 
$u\not\equiv -\infty$ be a subharmonic function on ${\mathbb{C}}$,  then
\begin{subequations}\label{10}
\begin{flalign}
{\sf B}_{U^+}(r) &\leq
\frac{7k\ln (ek)}{k-1} \bigl({\sf T}_{U}(r_0,kr)+{\sf C}_{U^+}(r_0)\bigr) 
\quad\text{for all\/ $0<r_0<r\in \mathbb R^+$  and $1<k\in {\mathbb{R}}^+$},
\tag{\ref{10}T}\label{inDl+0}
\\ 
{\sf B}_{|u|}(r) &\leq
\frac{11k\ln (ek)}{k-1} \bigl({\sf M}_{u^+}(kr)+{\sf C}_{u^-}(r_0)\bigr) \quad\text{for all\/ $0<r_0<r\in \mathbb R^+$  and $1<k\in {\mathbb{R}}^+$}.
\tag{\ref{10}M}\label{uM0}
\end{flalign}
\end{subequations}
 \end{corollary}
\begin{proof}
Let $E:=\overline D(r)$ in \eqref{1}. Then $\uplambda(E)=\pi r^2$,  
\begin{equation*}
\ln\frac{100kr^2}{\uplambda (E)}= \ln\frac{100k}{\pi}\leq \frac{7}{2}\ln ek, \quad 
{\sf B}_{U^+}(r)\overset{\eqref{Bu}}{=}\frac{1}{\uplambda(E)}\int_{E} U^+{\,{\rm d}} \uplambda,
\quad {\sf B}_{|u|}(r) \overset{\eqref{Bu}}{=}\frac{1}{\uplambda(E)}\int_{E} |u|{\,{\rm d}} \uplambda,
\end{equation*}
and by  \eqref{1} we obtain \eqref{10}.
\end{proof}
For a meromorphic function $F\not\equiv 0$ on ${\mathbb{C}}$, in the frame of  traditional notation \eqref{TN}, we denote
the average of $\ln^+ |F|$ over a  disc $\overline D(r)$  as 
\begin{equation}\label{BF}
{m^{[2]}}(r,F)\underset{r\in {\mathbb{R}}^+}{:=}\frac{2}{r^2}\int_0^r\left(\frac{1}{2\pi}\int_0^{2\pi}\ln^+ \bigl|F(te^{i\varphi})\bigr|{\,{\rm d}} \varphi \right)\, t{\,{\rm d}} t
\overset{\eqref{{Cu}C}}{=}\frac{2}{r^2}\int_0^r m(t,F)\, t\!{\,{\rm d}} t\overset{\eqref{Bu}}{=}{\sf B}_{\ln^+ |F|}(r).
\end{equation}

\begin{corollary}  Let, under the conditions of Theorem\/ {\rm \ref{th1}}, $F\not\equiv 0$ be a meromorphic function, and $f\not\equiv 0$ be an entire function.
Then, in the traditional notation  \eqref{{TN}M}--\eqref{{TN}N} and \eqref{BF}, we have
\begin{subequations}\label{1lfF}
\begin{flalign}
\int_{E} \ln^+ \bigl|F(z)\bigr|{\,{\rm d}} \uplambda(z)&\leq
\frac{2k}{k-1} \bigl( T(kr,F)-N(r_0,F)\bigr) \uplambda(E)
\ln\frac{100kr^2}{\uplambda (E)},
\tag{\ref{1lfF}T}\label{inDl+lfF}
\\ 
\int_{E} \Bigl|\ln \bigl|f(z)\bigr|\Bigr| {\,{\rm d}} \uplambda(z)& \leq
\frac{3k}{k-1} \bigl(\ln^+ M(kr,f)+m(r_0,1/f )
\bigr) \ln\frac{100kr^2}{\uplambda (E)}.
\tag{\ref{1lfF}M}\label{uMlfF}
\\
{m^{[2]}}(r,F)&\leq
\frac{7k\ln (ek)}{k-1}\bigl( T(kr,F)-N(r_0,F)\bigr),
\tag{\ref{1lfF}F}\label{inDl+0l}
\\ 
{m^{[2]}}(r,f)+{m^{[2]}}(r,1/f) &\leq
\frac{11k\ln (ek)}{k-1} \bigl({\sf M}_{u^+}(kr)+{\sf C}_{u^-}(r_0)\bigr) 
\tag{\ref{1lfF}f}\label{uM0l}
\end{flalign}
\end{subequations}
for all\/ $0<r_0<r\in \mathbb R^+$  and $1<k\in {\mathbb{R}}^+$.
\end{corollary}
\begin{proof} For $U:=\ln |F|$, we obtain  \eqref{inDl+lfF} by   \eqref{inDl+}, \eqref{TT}, \eqref{cs}, and also 
\eqref{inDl+0l} by \eqref{inDl+0}, \eqref{TT}, \eqref{BF}. 

For $u:=\ln |f|$ we obtain  \eqref{uMlfF} by   \eqref{uM}, \eqref{MM}, \eqref{cs}, and also 
 \eqref{uM0l} by  \eqref{uM0}, \eqref{MM}, \eqref{BF} since
${m^{[2]}}(r,f)+{m^{[2]}}(r,1/f)\overset{\eqref{BF}}{=}{\sf B}_{|\ln |f||}(r)$ for all $r\in {\mathbb{R}}^+$.
\end{proof}

\section{Lemmata and Proof of Theorem \ref{th1}}\label{Sec3}

\begin{lemma}\label{lem1} Let\/  $0\leq r<R<+\infty$,  $E\subset \overline D(r)$ be $\uplambda$-measurable, $U=u-v$  be a difference of subharmonic  functions $u, v\in {\sf sbh}_* \bigl(\overline D(R)\bigr)$,    
$\varDelta_v$ be the Riesz measure of $v$. Then 
\begin{equation}\label{inDR}
\int_{E} U^+{\,{\rm d}} \uplambda\leq \frac{1}{2}\Bigl(\frac{R+r}{R-r}
{\sf C}_{U^+}(R)+\varDelta_v^{{\text{\tiny\rm rad}}}(R)\Bigr) \uplambda(E)\ln\frac{(10R)^2}{\uplambda(E)}.
\end{equation}
\end{lemma}
\begin{proof}
For $w\in E\subset \overline D(r)$, by the Poisson\,--\,Jensen formula \cite[4.5]{Rans}, we have
\begin{multline*}
U(w)=\frac{1}{2\pi}\int_0^{2\pi}U(Re^{i\varphi}){\rm Re} \frac{Re^{i\varphi}+w}{Re^{i\varphi}-w}{\,{\rm d}} \varphi
-\int_{D(R)}\ln \Bigl|\frac{R^2-z\bar w}{R(w-z)}\Bigr|{\,{\rm d}} \varDelta_u(z) \\
+\int_{D(R)}\ln \Bigl|\frac{R^2-z\bar w}{R(w-z)}\Bigr|{\,{\rm d}} \varDelta_v(z)  
\leq \frac{R+r}{R-r}{\sf C}_{U^+}(R) +\int_{D(R)}\ln \frac{2R}{|w-z|}{\,{\rm d}} \varDelta_v(z) 
\end{multline*}
where the right-hand side of the inequality is positive. Hence, by integrating, we get 
\begin{multline}\label{intU}
\int_{E} U^+{\,{\rm d}} \uplambda \leq \frac{R+r}{R-r}{\sf C}_{U^+}(R)\uplambda(E)+
\int_{D(R)}\int_E\ln\frac{2R}{|w-z|}{\,{\rm d}} \uplambda (w){\,{\rm d}} \varDelta_v(z)\\
\leq \frac{R+r}{R-r}{\sf C}_{U^+}(R)\uplambda(E)+
\varDelta_v\bigl(\overline D(R)\bigr)
\int_E\ln\frac{2R}{|w-z|}{\,{\rm d}} \uplambda (w). 
\end{multline}
Denote by  $\uplambda_E$ the restriction of the Lebesgue measure  $\uplambda$ to $\uplambda$-measurable set $E\subset \overline D(r)$. 
Obviously, 
\begin{subequations}\label{lE1}
\begin{align}
{\rm supp\,} \uplambda_E&\subset \overline D(r)\subset D(R),
\tag{\ref{lE1}s}\label{{lE}s}\\
\uplambda_E(w,t)&\overset{\eqref{{murad}z}}{\leq} \pi t^2\quad\text{for each $w\in {\mathbb{C}}$ and $t\in {\mathbb{R}}^+$}, 
\tag{\ref{lE1}t}\label{{lE}t}\\
  \uplambda_E ({\mathbb{C}})&=\uplambda (E)\leq \pi r^2\leq \pi R^2.
\tag{\ref{lE1}E}\label{{lE}E}
\end{align}
\end{subequations} 
Consider the last integral in \eqref{intU}:
\begin{equation}\label{lI}
\int_E\ln\frac{2R}{|w-z|}{\,{\rm d}} \uplambda (w) =\int_{{\mathbb{C}}} \ln\frac{2R}{|w-z|}{\,{\rm d}} \uplambda_E (w)
\overset{\eqref{{lE}s}}{=}\int_0^{2R} \ln \frac{2R}{t}{\,{\rm d}} \uplambda_E (z;t)
\overset{\eqref{{lE}t}}{=}
\int_0^{2R} \frac{\uplambda_E (z;t)}{t}{\,{\rm d}} t.
\end{equation}
 By \eqref{{lE}E} we  have $\sqrt{\uplambda (E)}\overset{\eqref{{lE}E}}{\leq} \sqrt{\pi}r< 2r \leq 2R$.
From here we can split the last integral in \eqref{lI} into the sum of two positive integrals:
\begin{equation}\label{=E}
\int_0^{2R} \frac{\uplambda_E (z;t)}{t}{\,{\rm d}} t=
\int_0^{\sqrt{\uplambda (E)}} \frac{\uplambda_E (z;t)}{t}{\,{\rm d}} t+\int_{\sqrt{\uplambda (E)}}^{2R} \frac{\uplambda_E (z;t)}{t}{\,{\rm d}} t
\end{equation}
Using \eqref{{lE}t}, we have for first integral on the right-hand  side of this equality \eqref{=E} the estimate
\begin{equation}\label{=E1}
\int_0^{\sqrt{\uplambda (E)}} \frac{\uplambda_E (z;t)}{t}{\,{\rm d}} t\overset{\eqref{{lE}t}}{\leq} 
\int_0^{\sqrt{\uplambda (E)}} \frac{\pi t^2}{t}{\,{\rm d}} t=\frac{\pi}{2}\uplambda (E). 
\end{equation}
Using \eqref{{lE}E}, we have for second integral on the right-hand  side of  \eqref{=E} the estimate
\begin{equation}\label{=E2}
\int_{\sqrt{\uplambda (E)}}^{2R} \frac{\uplambda_E (z;t)}{t}{\,{\rm d}} t\leq 
\uplambda({\mathbb{C}})\int_{\sqrt{\uplambda (E)}}^{2R} \frac{\uplambda_E (z;t)}{t}{\,{\rm d}} t
\overset{\eqref{{lE}E}}{=} 
\frac12\uplambda(E) \ln \frac{4R^2}{\uplambda (E)}. 
\end{equation}
Thus, it follows from \eqref{=E1}--\eqref{=E2} and \eqref{=E} that
\begin{equation}\label{=E0}
\int_0^{2R} \frac{\uplambda_E (z;t)}{t}{\,{\rm d}} t\leq 
\frac{\pi}{2}\uplambda (E)+\frac12\uplambda(E) \ln \frac{4R^2}{\uplambda (E)} =
\frac{1}{2}\uplambda (E)\ln \frac{4e^{\pi}R^2}{\uplambda(E)}\leq 
\frac{1}{2}\uplambda (E)\ln \frac{(10R)^2}{\uplambda(E)},
\end{equation}
where
\begin{equation*}
\ln \frac{(10R)^2}{\uplambda(E)}\geq 2, \text{ since } \uplambda(E)\overset{\eqref{{lE}E}}{\leq} \pi R^2,
\end{equation*}
and estimate \eqref{=E0} together with \eqref{lI} and \eqref{intU} gives \eqref{inDR}.    
\end{proof}

\begin{lemma}[{\cite[Lemma 1]{KhaMst}}]\label{lem2} Let $\mu\in {\sf Meas\,}^+\bigl(\overline D(R)\bigr)$. Then
\begin{equation}\label{N}
\mu^{{\text{\tiny\rm rad}}}(r)\leq  \frac{R}{R-r}{\sf N}_{\mu}(r,R)\quad \text{for each $0\leq r\leq R$.}
\end{equation} 
\end{lemma}

\begin{lemma}\label{lem3} Let  $0< r<+\infty$, $0<b\in {\mathbb{R}}^+$,  $E\subset \overline D(r)$ be $\uplambda$-measurable, $U=u-v$  be a difference of subharmonic  functions $u, v\in {\sf sbh}_* \Bigl(\overline D\bigl((1+b)^2r\bigr)\Bigr)$. Then 
\begin{equation}\label{inDl}
\int_{E} U^+{\,{\rm d}} \uplambda \leq
\frac{2+b}{2b}
\Bigl(
{\sf C}_{U^+}\bigl((1+b)r\bigr)
+{\sf N}_{\varDelta_v}\bigl((1+b)r, (1+b)^2r\bigr)\Bigr)
\uplambda(E)\ln\frac{\bigl(10(1+b)r\bigr)^2}{\uplambda(E)}. 
\end{equation}
\end{lemma}

\begin{proof}[Proof] By Lemma \ref{lem1} 
with  $R:=(1+b)r$ we have 
\begin{equation*}
\int_{E} U^+{\,{\rm d}} \uplambda \leq \frac{1}{2}\Bigl(\frac{2+b}{b}
{\sf C}_{U^+}(R)+\varDelta_v^{{\text{\tiny\rm rad}}}\bigl((1+b)r\bigr)\Bigr) \uplambda(E)\ln\frac{\bigl(10(1+b)r\bigr)^2}{\uplambda(E)}.
\end{equation*}
Hence, by Lemma \ref{lem2} with $(1+b)^2r$ instead of $R$ and $(1+b)r$ instead of $r$, we obtain
\begin{equation*}
\int_{E} U^+{\,{\rm d}} \uplambda \leq \frac{1}{2}\frac{2+b}{b}
\Bigl({\sf C}_{U^+}\bigl((1+b)r\bigr)
+{\sf N}_{\varDelta_v}\bigl((1+b)r,(1+b)^2r\bigr) \Bigr) 
\uplambda(E)\ln\frac{\bigl(10(1+b)r\bigr)^2}{\uplambda(E)},
\end{equation*}
 which gives \eqref{inDl}.
\end{proof}
\begin{proof}[Proof of Theorem\/ {\rm  \ref{th1}}] 
We can assume that $U=u-v$ is the canonical representation of $U$. 
Consider a number $b>0$ such that $(1+b)^2=k$. By Lemma \ref{lem3}, we have
\begin{multline*}
\int_{E}U^+{\,{\rm d}} \uplambda \leq
\frac12\frac{\sqrt{k}+1}{\sqrt{k}-1}
\Bigl( {\sf C}_{U^+}(\sqrt{k}r)
+{\sf N}_{\varDelta_v}(\sqrt{k}r, kr)\Bigr)
\uplambda(E)\ln\frac{\bigl(10\sqrt{k}r\bigr)^2}{\uplambda(E)}
\\ 
\leq \frac{k}{k-1}
\Bigl( \underset{{\sf T}_{U}(r_0, \sqrt{k}r)}{\underbrace{{\sf C}_{U^+}(r_0, \sqrt{k}r)+{\sf N}_{\varDelta_v}(r_0, \sqrt{k}r)}}+{\sf C}_{U^+}(r_0)
+\underset{{\sf T}_{U}(r_0, kr)}{\underbrace{{\sf C}_{U^+}(r_0, kr)+{\sf N}_{\varDelta_v}(r_0, kr)}}\Bigr)\uplambda(E)\ln\frac{\bigl(10r\bigr)^2k}{\uplambda(E)}\\
\leq 
\frac{k}{k-1}
\Bigl( 2{\sf T}_{U}(r_0, kr)+{\sf C}_{U^+}(r_0)\Bigr)\uplambda(E)\ln\frac{\bigl(10r\bigr)^2k}{\uplambda(E)}
\leq \frac{2k}{k-1}
\Bigl( {\sf T}_{U}(r_0, kr)+{\sf C}_{U^+}(r_0)\Bigr)\uplambda(E)\ln\frac{\bigl(100kr^2}{\uplambda(E)}
\end{multline*}
and obtain  \eqref{inDl+}. 

If $u\in {\sf sbh}_*({\mathbb{C}})$, then the function  ${\sf M}_u^+$ is increasing, and
\begin{equation}\label{M+E}
\int_{E} u^+{\,{\rm d}} \uplambda \overset{\eqref{{Cu}M}}{\leq}
{\sf M}_{u^+}(r) \uplambda (E) \quad \text{for $E\subset \overline D(r)$.}
\end{equation}
For  $U_u:=0-u$, the difference  $0-u$ is the canonical representation of $\delta$-subharmonic non-trivial function $U_u$ and  we have
\begin{equation}\label{TCM}
{\sf T}_{U_u}(r,R)\overset{\eqref{T}}{=}  {\sf C}_{\sup\{0,u\}}(r,R)={\sf C}_{u^+}(r,R)\leq  {\sf C}_{u^+}(R)
\leq {\sf M}_{u^+}(R).
\end{equation}
Hence, by Theorem \ref{th1} in part \eqref{inDl+} for $U_u$ in the role of $U$, we obtain  
   
\begin{multline*}
\int_{E} (-u)^+{\,{\rm d}} \uplambda \overset{\eqref{inDl+}}{\leq}
\frac{2k}{k-1} \Bigl({\sf T}_{U_u}(r_0,kr)+{\sf C}_{U_u^+}(r_0)\Bigr)
\uplambda(E) \ln\frac{100kr^2}{\uplambda (E)}
\\
\overset{\eqref{TCM}}{\leq}
\frac{2k}{k-1} \Bigl({\sf M}_{u^+}(kr)+{\sf C}_{(-u)^+}(r_0)\Bigr)\uplambda(E) \ln\frac{100kr^2}{\uplambda (E)}. 
\end{multline*}
The latter together with \eqref{M+E} gives \eqref{uM}. 
\end{proof}


\begin{acknowledgments}
 The work was supported by a Grant of the Russian Science Foundation
(Project No. 18-11-00002).
\end{acknowledgments}



\end{document}